\documentclass[a4paper,11pt]{article}

\usepackage[english]{babel}

\usepackage{ae,aecompl}

\usepackage{graphicx}
\graphicspath{{Figures/}}

\usepackage{color}

\usepackage{amsmath}
\usepackage{amsthm}
\usepackage{amssymb}

\usepackage{psfrag}

\theoremstyle{plain}
\newtheorem{thm}{Theorem}
\newtheorem{prop}[thm]{Proposition}
\newtheorem{lemma}[thm]{Lemma}
\newtheorem{cor}[thm]{Corollary}

\theoremstyle{definition}

\theoremstyle{remark}

\newcommand{\refthm}[1]{Theorem~\ref{#1}}
\newcommand{\refprop}[1]{Proposition~\ref{#1}}
\newcommand{\reflemma}[1]{Lemma~\ref{#1}}
\newcommand{\refcor}[1]{Corollary~\ref{#1}}

\newcommand{\refsection}[1]{Section~\ref{#1}}

\newcommand{\reffig}[1]{Figure~\ref{#1}}
\newcommand{\junk}[1]{}

\newcommand{\twlog}{w.l.o.g.\ }

\newcommand{\vef}[1]{\mathop{\mathbf{P}_{#1}}}
\newcommand{\vf}[1]{\mathop{\mathbf{Q}_{#1}}}

\newcount\shortyear\newcount\shorthour\newcount\shortminute
\shorthour=\time\divide\shorthour by 60\shortyear=\shorthour
\multiply\shortyear by 60\shortminute=\time\advance\shortminute by
-\shortyear \shortyear=\year\advance\shortyear by -1900

\def\zeit{\number\shorthour:\ifnum\shortminute<10 0\number\shortminute
\else\number\shortminute\fi}

\def\ni{\noindent}

\def\PP{\mathbf{P}}
\def\cR{\mathcal{R}}

\def\Claim#1.{\medbreak\ni{\bf Claim #1.}\ }
\def\Fact#1.{\medbreak\ni{\bf Fact~#1.}\ }
\def\Example#1.{\medbreak\ni{\bf Example~#1.}\ }
\def\Case#1.{\medbreak\ni{\bf Case~#1.}\ }
\def\SubCase#1.{\medbreak\ni{\bf Subcase~#1.}\ }

\def\Proof{\ni{\sl Proof.}\ }
\def\qed{\hfill\fbox{\hbox{}}\medskip}
\def\qedclaim{\hfill$\triangle$\smallskip}

\def\ITEMMACRO #1 ??? #2 ???{\par\vskip4pt\noindent%
\hangindent=#2em\setbox0\hbox{#1\kern4pt}%
\ifdim\wd0<\hangindent\setbox0\hbox to\hangindent{\hss#1\kern7pt}\fi%
\box0\ignorespaces}
 
\def\Item(#1){\ITEMMACRO {\rm (#1)} ??? 1.8 ???}

\let\Bitem=\bItem
\def\BrackItem[#1]{\ITEMMACRO [#1] ??? 1.8 ???}

\gdef\SetFigFont#1#2#3#4#5{}
\newcount\magproz 
\newcount\vorne
\newcount\hinten
\newcount\zwischen
\def\path{Figures}

\def\calc_figscale#1{
   \magproz=#1
    \vorne=\magproz
    \divide\vorne by 100
    \hinten=\magproz
    \zwischen=\vorne
    \multiply\zwischen by 100
    \advance\hinten by -\zwischen
    \global\def\figscale{\the\vorne.\the\hinten}
}

\author{Stefan~Felsner \\
  {\small  Institut f\"{u}r Mathematik} \\
  {\small Technische Universit\"{a}t Berlin}\\
   \texttt{\small felsner@math.tu-berlin.de}
\and Johan~Nilsson\thanks{Supported by the EU Research Training Network 
 COMSTRU.} \thanks{This research was done while the author visited TU Berlin.} \\
    {\small MADALGO, BRICS} \\  
    {\small Department of Computer Science} \\
    {\small University of Aarhus} \\
    \texttt{\small johann@daimi.au.dk}
   }
\date{}
\title{On the Order Dimension of Outerplanar Maps}

\begin{document}
\enlargethispage{1cm}
\maketitle
\vbox{}
\vskip-15mm
\vbox{}
\begin{abstract} 
  Schnyder characterized planar graphs in terms of order dimension.
  Brightwell and Trotter proved that the dimension of the vertex-edge-face
  poset $\vef{M}$ of a planar map $M$ is at most four. In this paper
  we investigate cases where $\dim(\vef{M}) \leq 3$ and also
  where $\dim(\vf{M}) \leq 3$; here $\vf{M}$ denotes the vertex-face poset of $M$.
  We show:
  \smallskip

\ni{}\kern5pt$\bullet$\kern4pt{}If $M$ 
  contains a $K_4$-subdivision, then $\dim(\vef{M}) = \dim(\vf{M}) = 4$.
  \smallskip

\ni{}\kern5pt$\bullet$\kern4pt{}If 
   $M\kern-.1pt$ or the dual $M^*\kern-.1pt$ contains a 
   $K_{2,3}$-subdivision, then 
   \hbox{$\dim(\vef{M}) = 4$.}\smallskip

  \ni 
  Hence, a map $M$ with $\dim(\vef{M}) \leq 3$ must be outerplanar
  and have an outerplanar dual.
  We concentrate on the simplest class
  of such maps and prove that within this
  class $\dim(\vef{M}) \leq 3$ is equivalent to the existence of a
  certain oriented coloring of edges. This condition is easily checked and can 
  be turned into a linear time algorithm returning a 3-realizer.
  
  Additionally, we prove that if $M$ is 2-connected and
  $M$ and $M^*$ are outerplanar, then $\dim(\vf{M}) \leq 3$.
  There are, however, outerplanar maps with $\dim(\vf{M}) = 4$.
  We construct the first such example.
\end{abstract}

\section{Introduction}
This paper is about planar maps and the order dimension of posets
related to them. A \emph{planar map} $M = (G,D)$
consists of a finite planar multigraph $G$ and a plane drawing $D$ of
$G$. By a planar map $M$ we mean the combinatorial data 
given by the set $V$ of vertices, the set $E$ of edges, the set $F$ of faces
and the incidence relations between these sets.

The dual map $M^*$ of $M$ is defined as usual: there is a vertex $F^*$ 
in $M^*$ for each face $F$ in $M$, and an edge $e^*$ in $M^*$ for each 
edge $e$ of $M$, joining the dual vertices corresponding to the faces 
in $M$ separated by $e$ (if $e$ is a bridge, $e^*$ is a loop). 
Each vertex in $M$ will then correspond to a face of~$M^*$.

Most of the maps we consider in this paper are outerplanar.  We
differentiate between two notions of outerplanar maps.  A planar map
$M=(G,D)$ is \emph{weakly outerplanar} if $G$ is outerplanar, and
\emph{strongly outerplanar} if $G$ is outerplanar and $D$ is an
outerplane drawing of $G$, i.e., a plane drawing of $G$ where all the
vertices are on the boundary of the outer face. When it is clear from
the context, the qualifiers weakly and strongly will be omitted.

The dimension is a widely studied parameter of posets.
Since its introduction by Dushnik and Miller~\cite{DuMi41} in 1941,
dimension has moved into the core of combinatorics. There are 
close connections and analogies with the chromatic number of graphs
and hypergraphs. From the applications point of view, dimension is
attractive because low dimension warrants a small encoding complexity
of the poset. Trotter~\cite{Trotter} provides an extensive
introduction to the area. 

The vertex-edge-face poset $\vef{M}$ of a planar map $M$ is the
poset on the vertices, edges and faces of $M$ ordered by
inclusion. The vertex-face poset $\vf{M}$ of $M$ is the subposet
of $\vef{M}$ induced by the vertices and faces of~$M$.

   \calc_figscale{100}
    \begin{figure}[htb]
    \centerline{\input{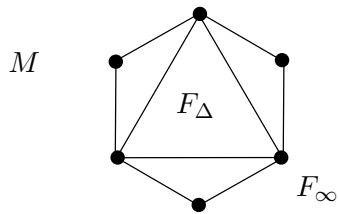}}
    \caption{A planar map. Two faces are labeled, $F_\infty$ is the {\em outer face}.\label{fig:map_ex}}
    \end{figure}
    

   \calc_figscale{55}
    \begin{figure}[htb]
    \centerline{\input{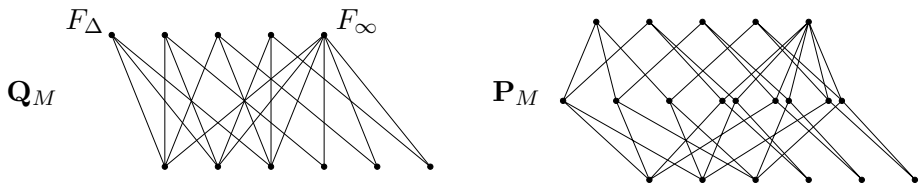}}
    \caption{The vertex-face and the vertex-edge-face posets of the map from \reffig{fig:map_ex}.\label{fig:vf+vef}}
    \end{figure}
    

Note that if $M$ is connected, the vertex-edge-face poset $\vef{M^*}$
of the dual map is just the dual poset $(\vef{M})^*$ (i.e., $x < y$
in $(\vef{M})^*$ if and only if $y < x$ in $\vef{M}$). The same
observation is true for $\vf{M}$.

The \emph{dimension} $\dim(\PP)$ of a poset $\PP$ is the
minimum number $t$ such that $t$ is the intersection of $t$ linear
orders on the same ground set. From this it follows
that if $Q$ is an induced subposet of $P$, then
$\dim(\vf{}) \leq \dim(\vef{})$. In particular this implies
$\dim(\vf{M}) \leq \dim(\vef{M})$ for every map $M$.

The background to our investigations lies in Schnyder's
characterization of planar graphs in terms of
dimension~\cite{Schnyder89} and in the two theorems of Brightwell and
Trotter stated below.  A simpler proof of Theorem~\ref{thm:bt2} can be
found in~\cite{FelZick07}.

\begin{thm}[Brightwell and Trotter \cite{BT97}]
  \label{thm:bt1}
    If $M$ is a planar map, then $\dim(\vef{M}) \leq 4$.
\end{thm}

\begin{thm}[Brightwell and Trotter \cite{BT93}]
  \label{thm:bt2}
    If $M$ is a 3-connected planar map, then $\dim(\vf{M}) = 4$.
\end{thm}

It was observed by Li~\cite{Li} that the characterization of maps with
$\dim(\vef{M}) = 2$ given in the concluding section of \cite{BT97} is
incomplete but the full characterization exists, the paper is in 
preparation.
The case of dimension is 2 is also well-understood algorithmically: There are fast
algorithms to test whether a poset $\PP$ is of dimension 2, see
e.g.~\cite{Moe}. 

For $\dim(\vef{M}) \geq 3$ it is sufficient that $M$ has a vertex of
degree 3.  To test whether an arbitrary poset has dimension equal to~3
is known to be NP-complete~\cite{Yannakakis82}.  A major open problem
in the area is to determine the complexity of the dimension 3 problem
for orders of height 2. Even the special case where the order of
height 2 is the vertex-face poset of a planar map remains
open. Motivated by these algorithmic questions we approach the
problems of characterizing the maps $M$ with $\dim(\vef{M}) \leq 3$
and the maps with $\dim(\vf{M}) \leq 3$. These characterization
problems have earlier been posed by Brightwell and Trotter
\cite{BT97}.

\subsection{Our contributions}
In \refsection{sec:lower_bound} we prove that for $\dim(\vf{M}) \leq 3$
it is necessary
that $M$ is $K_4$-subdivision free. For $\dim(\vef{M}) \leq 3$ an additional
necessary condition is that both $M$ and $M^*$ are $K_{2,3}$-subdivision free. 
This means that if $\dim(\vef{M}) \leq 3$, then both $M$ and $M^*$
are outerplanar. 
%

In \refsection{sec:path_like}, we study the simplest class of maps $M$
such that $M$ and $M^*$ are outerplanar. We call these
maps {\it path-like}. For maximal path-like maps we prove that $\dim(\vef{M})
\leq 3$ is equivalent to the existence of a special oriented coloring
of the interior edges and characterize the path-like maps which admit such
a coloring. The characterization is turned into a linear time algorithm that
generates a 3-realizer, i.e., three linear extensions whose intersection 
is $\vef{M}$, or returns the information that $\dim(\vef{M}) \geq 4$.

Finally, in \refsection{sec:vf}, we prove that if $M$ is 2-connected and
$M$ and $M^*$ are outerplanar, then $\dim(\vf{M}) \leq 3$. We also
present a strongly outerplanar map with a
vertex-face poset of dimension 4. The example, a
maximal outerplanar graph with 21 vertices, is quite large.
We provide some arguments which indicate that our 
example is not far from being as small as possible.

\subsection{Tools from dimension theory}

In this section we recall some facts from the dimension theory 
of finite posets. The reader is referred to Trotter's monograph 
\cite{Trotter} for additional background and references.

If $\PP$ is a finite poset, a family $\cR = \{L_1, L_2,
\ldots , L_t\}$ of linear extensions of
$\PP$ is called a \emph{realizer} of $\PP$ if $P = \cap
\cR$, i.e. $x<y$ in $\PP$ if and only if $x<y$ in $L$
for all $L \in \cR$. The dimension of $\PP$ is the
minimum cardinality of a realizer of $\PP$.

A \emph{critical pair} is a pair of incomparable elements $(a,b)$ 
such that $x<b$ if $x<a$ and $y>a$ if $y>b$ for all $x,y \in 
\PP$. A family of linear extensions $\cR$ 
of $P$ is a realizer of $\PP$ if and only if each 
critical pair $(a,b)$ is \emph{reversed} in some linear extension $L\in \cR$, 
i.e., $b < a$ in $L$. An incomparable min-max pair, i.e., a pair of
incomparable elements $(a,b)$ where $a$ is a minimal element  and $b$
is a maximal element of $\PP$, is always critical.
 
An alternating cycle  is a sequence of critical pairs
$(a_0,b_0), \ldots, (a_k, b_k)$ such that $a_i < b_{(i+1 \mod (k+1))}$
for all $i=0,\ldots,k$. A fundamental result is that $\dim(\PP)
\leq t$ if and only if there exists a $t$-coloring of the critical pairs of
$\PP$ such that no alternating cycle is monochromatic.

In the following example we illustrate how these facts can be combined to
determine the dimension of a specific incidence order.
\smallskip

\noindent
{\bf Example:} Let $M$ be the planar map of the complete graph $K_4$.
Every vertex has a single non-incident face, hence, there are
these four incomparable min-max pairs in $\vf{M}$. These are all the critical pairs.
Any two of these
critical pairs form an alternating cycle. Therefore, the hypergraph of alternating
cycles is again a $K_4$ and has chromatic number 4. This shows that
$\dim(\vf{M})=4$.

\section{Vertex-edge-face posets of dimension at most 3}
\label{sec:lower_bound}

From \refthm{thm:bt2}, we know that $\dim(\vf{M})=4$ for every
3-connected map~$M$. We show that this excludes
$K_4$-subdivisions from being contained in $M$ if $\dim(\vf{M}) \leq
3$. 

\begin{thm}\label{thm:k4_lower_bound}
  Let $M$ be a planar map that contains a subdivision of $K_4$. Then
  $\dim(\vf{M}) > 3$.
\end{thm}

\begin{proof}
  We will prove that $\vf{M}$ has the vertex-face poset of some 3-connected planar 
  map as a subposet, and then apply the Brightwell-Trotter Theorem.
  This is essentially done in two steps: first 1-vertex cuts and then
  2-vertex cuts are removed. 

  A $K_4$-subdivision in $M$ will be contained in a 2-connected component
  of $M$. The vertex-face poset of a 2-connected component of
  $M$ is an induced subposet of $\vf{M}$. Hence, we can assume that
  $M$ is 2-connected.

   \calc_figscale{60}
    \begin{figure}[htb]
    \centerline{\input{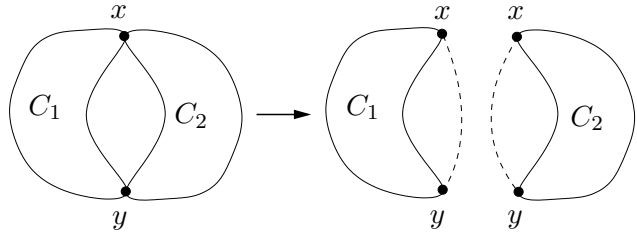}}
    \caption{Removing a separating pair.\label{fig:sep_pair}}
    \end{figure}

  Now, consider a 2-vertex cut $\{x,y\}$. There must be two components $C_1$ and $C_2$ such 
  that removing $x$ and $y$ separates $C_1$ from $C_2$. We create two
  new maps by replacing one of the two components by an edge
  $\{x,y\}$, see \reffig{fig:sep_pair}.
  The vertex-face posets of
  these new maps are subposets of $\vf{M}$. Furthermore, if $M$
  contains a $K_4$-subdivision, one of the new maps must contain a $K_4$-subdivision.
\end{proof}

For vertex-edge-face posets, we provide another criterion which forces 
dimension 4. If $M$ contains a $K_{2,3}$-subdivision, then $\dim(\vef{M}) =4$. 
Before analyzing the general situation we deal with the 
simple case where $M$ is, actually, a subdivision of $K_{2,3}$.

\begin{prop}\label{prop:k23_lower_bound_basic}
  Let $M$ be a planar drawing of a subdivision of $K_{2,3}$. Then $\dim(\vef{M}) > 3$.
\end{prop}

\begin{proof}
 Let $x$ and $y$ be the two vertices of degree 3 in $M$, and
  let $P_1$, $P_2$ and $P_3$ be the three $x$-$y$
  paths. The map has three faces $F_1$, $F_2$ and $F_3$.
  The vertex closest to $y$ in the path $P_i$ is denoted $v_i$,
  see \reffig{fig:k23-new}. 

  Suppose $\{L_1, L_2, L_3\}$ is a realizer of $\vef{M}$. By symmetry, 
  we may assume that $y > x$ in $L_1$ and $L_2$ and $x > y$ in $L_3$. 
  The edge $\{v_1,y\}$ can go below $x$ only in $L_3$.
  In $L_3$ we thus have $v_1$ below $\{v_1,y\}$ below $x$ 
  below $F_1$, $F_2$ and $F_3$. In the same way, we obtain
  that $v_2$ and $v_3$ are below $F_1$, $F_2$ and $F_3$ in
  $L_3$. Hence, none of the three critical pairs of the the subposet $\mathbf{S}$
  induced by $v_1$, $v_2$, $v_3$, $F_1$,
  $F_2$ and $F_3$ in $\vef{M}$ can be reversed in $L_3$. However, $\mathbf{S}$ is a
  crown with $\dim(\mathbf{S}) = 3$. This shows that $\{L_1, L_2, L_3\}$ is not
  a realizer of $\vef{M}$. Hence, $\dim(\vef{M}) > 3$.
\end{proof}

   \calc_figscale{80}
    \begin{figure}[htb]
    \centerline{\input{\path/k23-new.pstex_t}}
    \caption{\label{fig:k23-new}}
    \end{figure}
    VC
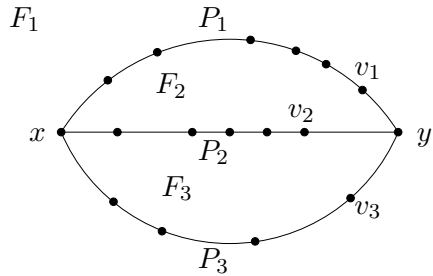
{The three paths $P_1$, $P_2$ and $P_3$ partition the map into 3 regions.}

For the general case, where $M$ only contains a
$K_{2,3}$-subdivision, we have to use a more sophisticated technique.
We illustrate this technique  with an alternative proof of the simple case.

\noindent{\em Second proof of \refprop{prop:k23_lower_bound_basic}.}
  The vertices and edges of
  path $P_i$ all belong to $F_i \cap F_{i+1}$ (cyclically). 
  Hence, each path $P_i$ induces a fence of the form $x < e_0 > u_1 < e_1 > \ldots
  > u_s < e_s > y$ between $x$ and $y$ in
  $\vef{M}$ such that all maximal elements are below $F_{i}$ and
  $F_{i+1}$. These three fences are mutually disjoint except for $x$ and $y$.

  Suppose $\{L_1, L_2, L_3\}$ is a realizer of $\vef{M}$. By symmetry,
  we may assume that $y > x$ in $L_1$ and $L_2$ and $x > y$ in
  $L_3$. Now, consider the fence induced by $P_i$, $i\in 1,2,3$; see
  \reffig{fig:new-fence}. 

  The edge $\{v_i,y\}$ must be below $x$ in $L_3$, hence $v_i$ is below $x$ in
  $L_3$. Let $w_i$ be the last vertex encountered when traversing the
  path $P_i$ from $y$ to $x$ which is below $x$ in $L_3$,
  and let $e_i$ be the edge leaving $w_i$ in direction of $x$. The choice of
  $w_i$ implies that $e_i$ is above $x$ and $y$ in $L_3$. Since $e_i$ has to go
  below $y$ somewhere there is an index $j_i \in \{1,2\}$ such that $e_i$ and 
  thus $w_i$ go below $y$ in $L_{j_i}$. 

  Two of the three indices $j_1,j_2,j_3$ must be equal, so we 
  can \twlog assume that $w_1$ and $w_2$ are below $y$ in $L_2$.
  Recalling that $w_1$ and $w_2$ are below $x$ in $L_3$ we
  conclude that $w_1$ and $w_2$ are below all faces $F_j$
  in $L_2$ and~in~$L_3$.

  Now, none of the critical pairs of the subposet {\bf 2+2} of
  $\vef{M}$ induced by $w_1$, $w_2$, $F_1$ and $F_3$ are reversed in
  $L_2$ or $L_3$. But $\dim(${\bf 2+2}$)=2$, so the critical pairs of $\mathbf{Q}$ cannot
  be reversed in $L_1$ alone. Hence $\{L_1,L_2,L_3\}$ cannot be a realizer
  of $\vef{M}$.
\qed

   \calc_figscale{65}
    \begin{figure}[htb]
    \centerline{\input{\path/new-fence.pstex_t}}
    \caption{\label{fig:new-fence}}
    \end{figure}
    VC
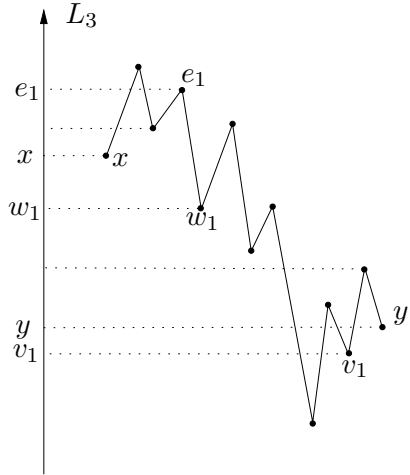
{The fence of path $P_1$ and $L_3$.}

We now move on to the slightly more complicated case where $M$ only
contains a subdivision of $K_{2,3}$.

\begin{thm}\label{thm:k23_lower_bound}
  Let $M$ be a planar map such that $M$ contains a subdivision of $K_{2,3}$. 
  Then $\dim(\vef{M}) > 3$.
\end{thm}

\begin{proof}
  If $M$ contains a subdivision of $K_4$, the conclusion of the lemma
  follows from \refthm{thm:k4_lower_bound}.  We thus can assume
  that $M$ contains no subdivision of $K_4$. 

  Let $x$ and $y$ be the degree 3 vertices in the subdivision of
  $K_{2,3}$.  Our goal is to find at least three mutually disjoint
  fences $\mathbf{T}_i$ between $x$ and $y$, and a set of faces $F_i$
  such that $x,y \in F_i$ and each minimal element in $\mathbf{T}_i$
  is below $F_i$ and $F_{i+1}$.

  Given fences $\mathbf{T}_i$ and faces $F_i$ as described we can
  continue as in the previous proof: Assume a realizer
  $\{L_1,L_2,L_3\}$ such that $y > x$ in $L_1$ and $L_2$ and $x > y$
  in $L_3$. In each fence $\mathbf{T}_i$ we find a minimal element
  $w_i$ which is below $x$ in $L_3$ and below $y$ in some $L_{j_i}$,
  $j_i\in \{1,2\}$. Since $i\geq 3$ there are indices $a$ and $b$ with
  $j_a = j_b$, and we can \twlog let $j_a = j_b = 1$. Let $a'\in \{a,a+1\}$ and
  $b'\in \{b,b+1\}$ be such that $w_b\not\in F_{a'}$ and $w_a\not\in
  F_{b'}$. Hence, $w_a,F_{a'},w_b,F_{b'}$ induce a {\bf 2+2}. The
  critical pairs of this {\bf 2+2} are not reversed in 
  $L_3$ nor in $L_1$, and they can't both be reversed in $L_2$. This is
  in contradiction to the assumption that $\{L_1,L_2,L_3\}$ is a
  realizer. Hence, $\dim(\vef{M}) >3$. 

  It remains to show how to determine appropriate fences
  $\mathbf{T}_i$.  Consider a maximal set $P_0,P_1,\ldots,P_k$ of
  pairwise internally disjoint paths from $x$ to~$y$. Clearly, $k\geq
  2$. Choose the numbering corresponding to the cyclic order of their
  first edges at $x$.  For $i=1,..,k+1$ let $R_i$ be the bounded area
  between $P_{i-1}$ and $P_i$, with $P_{k+1} = P_0$. 

  \Claim A. 
   If $M$ contains no subdivision of $K_4$, then every region
    $R_i$ contains a face $F_i$ with $x$ and $y$ and interior vertices
    of $P_{i-1}$ and $P_i$ on its boundary.
  \smallskip

  \ni The maximality of the family $P_0,P_1,\ldots,P_k$ implies that
  in $R_i$ there is a face $F_i$ that has a nonempty intersection with
  the interior of both $P_{i-1}$ and $P_i$.  Next we prove that this
  face $F_i$ contains $x$ and $y$.  Otherwise, the cycle consisting of
  $P_{i-1}$ and $P_{i}$ has a chordal path separating $x$ from $y$. This
  path, together with $P_{i-1}$, $P_{i}$ and some $P_j$,
  $j\not\in\{i-1,i\}$ is a subdivision of $K_4$ in $M$.  \qedclaim

  Since a subdivision of $K_4$ already implies that $\dim(\vef{M}) =4$
  we continue with the assumption that there is a face $F_i$ with the specifiesd
  properties for all $i$.

  Let $u$ and $w$ be vertices of $P_i$ such that $u$ is closer to $x$
  than $w$ and $(u,w)\neq(x,y)$.  A \emph{shortcut} between $u$ and
  $w$ over $P_i$ is a path from $u$ to $w$ which is internally
  disjoint from $P_i$.  Two shortcuts, $\{u_1,w_1\}$ and $\{u_2,w_2\}$
  are \emph{crossing} if their order along $P_1$ is either
  $u_1,u_2,w_1,w_2$ or $u_2,u_1,w_2,w_1$. In particular this requires
  the four vertices to be pairwise different.

   \calc_figscale{95}
    \begin{figure}[htb]
    \centerline{\input{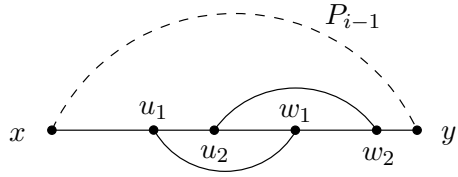}}
    \caption{Crossing shortcuts.\label{fig:pathintersection}}
    \end{figure}
    

  \Claim B. There is no crossing pair of shortcuts on $P_i$.
  \smallskip

  \ni
  Otherwise, the four vertices of the two shortcuts are the degree three vertices of a
  subdivision of $K_4$. This subdivision of $K_4$ is formed by the shortcuts together
  with $P_i$ and $P_{i-1}$. See \reffig{fig:pathintersection}.
  \qedclaim

  Let $V_i$ be the set of all vertices of $P_i$
  that are contained in $F_i \cap F_{i+1}$.

  \Claim C. Two consecutive vertices $u$ and $w$ in $V_i$ either are the two endpoints of 
  an edge or there exists a face $F$ such that $F \cap V_i = \{u,w\}$.
  \smallskip

   \calc_figscale{95}
    \begin{figure}[htb]
    \centerline{\input{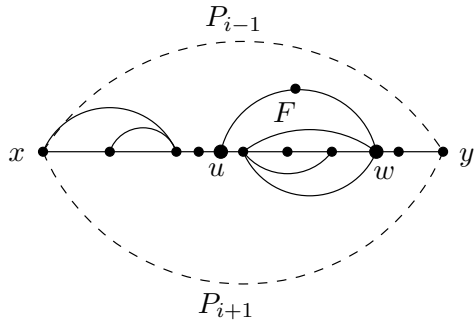}}
    \caption{The common face $F$ of $u$ and $w$.\label{fig:path_common_face}}
    \end{figure}
    

  \ni
  Suppose $\{u,w\}$ is not an edge. From Claim B it follows that there is a shortcut 
  $\{u,w\}$ over $P_i$. Essentially the same proof as for Claim A shows that 
  the subregion bounded  by $P_i$ and the shortcut between $u$ and $w$ 
  contains a face $F$ with $u,w\in F$; 
  otherwise, there is a $K_4$-subdivision.
  \qedclaim.

  The fence $\mathbf{T}_i$ consists of $V_i$ (the set of minimal elements) 
  and edges, respectively faces, over consecutive pairs of vertices in $V_i$.
  The existence of a $K_{2,3}$-subdivision between $x$ and $y$ implies that
  at least three of the fences $\mathbf{T}_0,\mathbf{T}_1,\ldots,\mathbf{T}_k$
  are nontrivial, i.e., have minima different from $x$ and~$y$. These 
  fences can be used to conclude the proof.
\end{proof}

\begin{thm}\label{thm:<4outerp}
If $\dim(\vef{M}) \leq 3$, then $M$ and $M^*$ are both weakly outerplanar.
\end{thm}

\begin{proof}
From \refthm{thm:k4_lower_bound} and
\refthm{thm:k23_lower_bound}, we know that if $\dim(\vef{M}) \leq
3$ then $M$ contains neither a 
$K_4$-subdivision nor a $K_{2,3}$-subdivision.
This is equivalent to saying that the graph $G$ corresponding to the map $M$ is
outerplanar. Since $\vef{M}$ and $\vef{M^*}$ are dual orders and, hence
have the same dimension, the same necessary condition for $\dim(\vef{M}) \leq
3$ applies to $M^*$.
\end{proof}

Note that testing if $M$ and $M^*$ are weakly outerplanar, i.e., if
the corresponding graphs are outerplanar, can be done in
linear time \cite{Mitchell}.

\section{Path-like maps and permissible colorings}
\label{sec:path_like}

From \refthm{thm:<4outerp} we know that if $\dim(\vef{M}) \leq 3$, both $M$
and $M^*$ are weakly outerplanar. In this section we study the order dimension of 
2-connected maps $M$, such that $M$ is strongly outerplanar and $M^*$ is weakly 
outerplanar. 

A 2-connected component of an outerplanar map $M$ has a Hamilton cycle. 
If the graph of $M$ is simple, the Hamilton cycle is unique. This
yields a natural partition of the edges of $M$ into {\em cycle edges} and
{\em chordal edges}. The restriction of the dual graph to the graph induced
by the vertices corresponding to bounded faces is called the {\em interior dual}.
For a strongly outerplanar map, the edges of the interior dual are just the 
dual edges of the chordal edges.

We say that a simple 2-connected outerplanar map $M$ is \emph{path-like} if and only if
the interior dual of $M$ is a simple path. Note that this implies
that the Hamilton cycle is the boundary of the outer face $F_\infty$, i.e.
that $M$ is strongly outerplanar. Since the interior dual is a path, it follows 
that $M^*$ is weakly outerplanar. On the other hand, if $M$ is a 2-connected 
strongly outerplanar map and $M^*$ is weakly outerplanar, the interior dual of 
$M$ must be a simple path. Hence, $M$ is path-like if and only if $M$ is a 2-connected
outerplanar map such that $M^*$ is weakly outerplanar. 

From \refthm{thm:<4outerp} it follows that if $M$ is a 2-connected strongly 
outerplanar map with $\dim(\vef{M}) \leq 3$, $M$ must be path-like. 
We can also prove something slightly stronger.

\begin{prop}
\label{prop:pathlike}
Let $M$ be a simple 2-connected planar map with $\dim(\vef{M}) \leq 
3$. The map $M'$ 
obtained by moving all the chordal edges of $M$ into the interior of
the Hamilton cycle is path-like.
\end{prop}

\Proof  Suppose not. Then the interior dual of $M'$ contains a vertex of
  degree at least 3, and hence its dual $(M')^*$ contains a
  subdivision of $K_{2,3}$ with the dual of one degree 3 vertex inside
  the Hamilton cycle $H$ and the dual of the other outside. We proceed
  to show that we can move the necessary chordal edges outside, one by
  one, to create $M$ without destroying the $K_{2,3}$-subdivision in
  the dual.

  We do this as follows: let $M' = M_0,M_1,\ldots,M_k = M$ be a
  sequence of maps such that $M_{j+1}$ is obtained from $M_j$ by
  moving a chordal edge from the inside to the outside of the
  Hamilton cycle.  The proposition is implied by the following lemma.

  \begin{lemma}
  For each map $M_i$, $i=0,1,\ldots,k$, the dual map $M_i^*$
  contains a $K_{2,3}$-subdivision such that $H^*$ separates the
  two vertices of degree 3.  
  \end{lemma}

  \Proof We prove the lemma by induction on $i$. 
  We have already seen that the statement is true for $M_0$.

  Suppose the statement is true for $M_i$. Let $e=\{u,v\}$ be the edge that
  has to be moved to the outside of $H$ to get from $M_i$ to
  $M_{i+1}$.  Let $F^*$ and $G^*$ be the degree 3 vertices in a
  $K_{2,3}$-subdivision in $M_i^*$, where $F$ is inside $H$. We
  construct a new map $M_i'$, by adding the edge $e'=\{u,v\}$ to $M_i$
  outside $H$, see~\reffig{fig:pathlike_k23}. Note that $F$ and $G$
  must be on the same side of the cycle $\{e,e'\}$, since otherwise
  $\{e^*,(e')^*\}$ is a 2-edge cut in $(M'_i)^*$ separating $F^*$ from
  $G^*$; this is impossible as $F^*$ and $G^*$ are the degree 3 vertices in a
  $K_{2,3}$-subdivision.

   \calc_figscale{100}
    \begin{figure}[htb]
    \centerline{\input{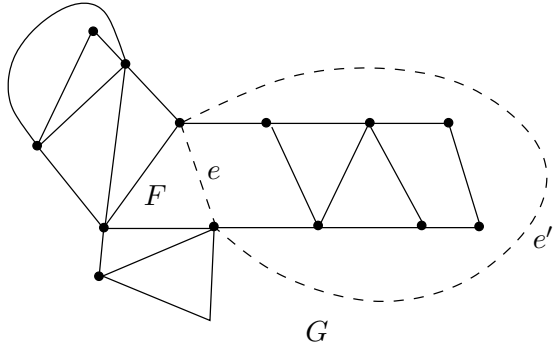}}
    \caption{The map $M_i'$ is constructed by adding $e'$ to $M_i$.\label{fig:pathlike_k23}}
    \end{figure}
    

  \Claim A. Let $P^*$ be a simple $F^*$-$G^*$ path in $(M_i')^*$. If
  $e^* \in P^*$, then $P^*$ consists of at least 3 edges.
  \smallskip

  \ni Suppose $e^*\in P^*$. Since $P^*$ is simple, and $F$ and $G$ are
  on the same side of the cycle $\{e,e'\}$, it follows that $(e')^*
  \in P^*$.  Moreover, the dual $H^*$ of the Hamilton cycle $H$ separates
  $F^*$ from $G^*$, so $P^*$ must contain the dual of a cycle edge.
  But neither $e$ nor $e'$ are cycle edges, so the claim follows.
  \qedclaim
  
  \ni Now, $M_{i+1}$ is obtained by removing $e$ from $M_i'$. In
  $(M_i')^*$, this corresponds to the contraction of edge $e^*$.  If
  $e^*$ is not on any $F^*$-$G^*$ path, then $M^*_{i+1}$ contains a
  $K_{2,3}$-subdivision. On the other hand, an $F^*$-$G^*$ path
  containing $e^*$ has at least 3 edges, hence, the contraction of $e^*$ cannot
  destroy the $K_{2,3}$-subdivision. This completes the proof of the lemma and the
  proposition.
\qed

For later reference we restate the essence of the proof:

\begin{cor}\label{cor:k23-free}
Let $M'$ be obtained from a simple weakly outerplanar map $M$ by
flipping all chordal edges to the interior of the Hamilton cycle.  If
$(M')^*$ contains a $K_{2,3}$-subdivision, then so does $M^*$.
\end{cor}

In the rest of this section, we consider {\em maximal path-like maps},
i.e., path-like maps where all interior faces are triangles.  Consider
a triangle of a maximal path-like map $M$. Each of the three vertices
forms a critical pair with a face or edge that is above the other two
vertices of the triangle. In \reffig{fig:triangle1} these critical
pairs are $(u,F_{u})$, $(v,F_{v})$, $(w, e_{w})$.  Note that the
critical pairs associated to a triangle are in bijection to its
angles.

Suppose $\dim(\vef{M})=3$ and $L_1,L_2,L_3$ is a realizer.  From this
we obtain a 3-coloring of the of inner angles of the map: For every
angle choose a color $i$ such that the critical pair corresponding to
the angle is reversed in $L_i$.  For convenience we interchangably use
1, 2, 3 and red, green, blue as names for the three colors in the angle
coloring.

   \calc_figscale{45}
    \begin{figure}[htb]
    \centerline{\input{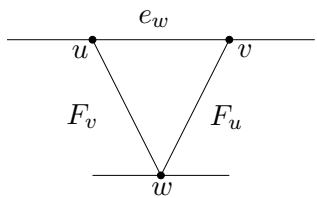}}
    \caption{The critical pairs of a triangle.\label{fig:triangle1}}
    \end{figure}
    

We go on to prove some properties of such an angle 3-coloring
for a maximal path-like map $M$ with $\dim(\vef{M}) \leq 3$.
In the following we will use $F_u$ as the default name for the
edge or face in the critical pair with $u$. For example in the proof
of the next lemma at least one of $F_u,F_v,F_w$ must be an edge 
since the map $M$ is path like.

\begin{lemma}
\label{lemma:angle_color}
No two angles in a triangle can have the same color.
\end{lemma}
\begin{proof}
Consider a triangle with the angle coloring described above. Any two
of the critical pairs $(u,F_{u})$, $(v,F_{v})$, $(w, F_{w})$ form an
alternating cycle. Hence, no two pairs can be reversed in the same
linear extension.
\end{proof}

\begin{lemma}
\label{lemma:edge_color}

Let $e=\{a,b\}$ be a chordal edge. The four angles $\alpha_{\ell}$,
$\alpha_r$, $\beta_{\ell}$ and $\beta_r$ incident on $e$ at $a$ and
$b$ are colored such that all three colors are used, and one of the
pairs $(\alpha_{\ell},\alpha_r$) or $(\beta_{\ell},\beta_r)$ is
monochromatic.
   \calc_figscale{70}
    \begin{figure}[htb]
    \centerline{\input{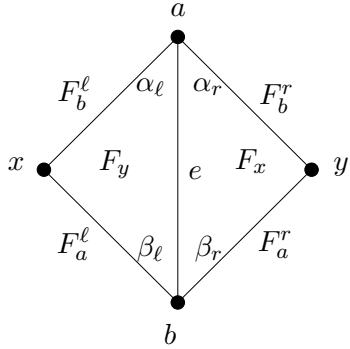}}
    \caption{Colors and critical pairs around a chordal edge.\label{fig:new-edge-color}}
    \end{figure}
\end{lemma}

\begin{proof}
We refer to \reffig{fig:new-edge-color}.
Suppose $\alpha_{\ell}=1$ and
$\alpha_r=2$. This implies that in $L_1$ we have
$F_b^\ell$ and $F_b^r$ above $a$ above $F_a^{\ell}$ above $b$.
In $L_2$ we have the same order with $F_a^{r}$ taking the role of $F_a^{\ell}$.
Hence $b$ has to be above both $F_b^\ell$ and $F_b^r$ in $L_3$ which
is equivalent to $\beta_\ell=\beta_r=3$.  

Suppose both pairs of angles have the same colors, say
$\alpha_{\ell}=\alpha_r=1$ and $\beta_\ell=\beta_r=2$
Then the third angle in both triangles (at $x$ and $y$,
respectively) must have color 3. This induces a
monochromatic alternating cycle $(x,F_{x})$, $(y,F_{y})$,
a contradiction.
\end{proof}

By \reflemma{lemma:edge_color} we can encode the angle coloring as
an oriented coloring of the chordal edges: each chordal edge gets the
color that appears twice around it and is oriented towards the
endpoint where this happens.

The \emph{orientation} of an interior triangle is either
\emph{clockwise} or \emph{counterclockwise} depending on the cyclic
reading which shows the colors 1,2,3 in this order. \reflemma{lemma:angle_color}
implies that the orientation of interior triangles is defined.

   \calc_figscale{65}
    \begin{figure}[htb]
    \centerline{\input{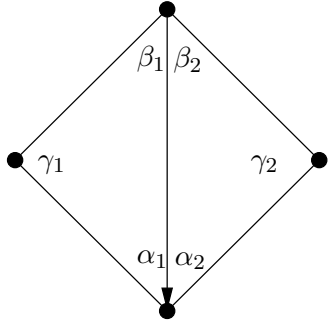}}
    \caption{The triangles must have the same orientation.\label{fig:same_orient}}
    \end{figure}
    

\begin{lemma}
\label{lemma:same_orient}
All interior triangles have the same orientation.
\end{lemma}

\begin{proof}
This is a direct consequence of \reflemma{lemma:edge_color} and
\reflemma{lemma:angle_color}. 
Referring to \reffig{fig:same_orient} we discuss one of the cases:
Suppose the left triangle in the figure is counterclockwise, i,e.,
$(\alpha_1,\beta_1,\gamma_1)= 
(i,i+1,i+2)$. The orientation of the edge implies $\alpha_2=i$, with
\reflemma{lemma:edge_color} we get $\beta_2=i+2$. This shows that
the right triangle is counterclockwise as well. 
\end{proof}
 
\begin{lemma}
  \label{lemma:chordal_edge}
  If $c$ is the color of a chordal edge $e$, then $e> F_{\infty}$ in
  $L_c$.
\end{lemma}
\begin{proof}
  Again referring to \reffig{fig:same_orient} we observe that 
  $\gamma_1\neq\gamma_2$  and
  $c \neq \gamma_1,\gamma_2$ by \reflemma{lemma:same_orient}. 
  Hence, $e < F_{\infty}$ in  $L_{\gamma_1}$ and $L_{\gamma_2}$.
  Therefore $e> F_{\infty}$ in $L_c$.
\end{proof}

   \calc_figscale{65}
    \begin{figure}[htb]
    \centerline{\input{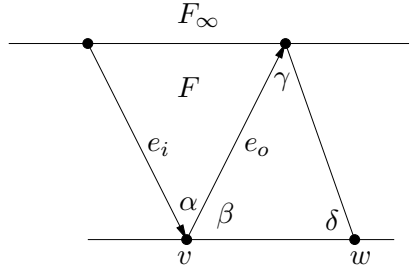}}
    \caption{A vertex must be either a sink or a source.\label{fig:in-out_edge}}
    \end{figure}
    

\begin{lemma}
\label{lemma:sink_or_source}
  Each vertex is either a sink or a source w.r.t.~the orientation of the chordal edges.
\end{lemma}
\begin{proof}
  Suppose that there is a triangle $F$ and a vertex $v$ such that $F$
  has two chordal edges $e_i$ and $e_o$ meeting at $v$, such that $e_i$
  is incoming and $e_o$ is outgoing at
  $v$. See \reffig{fig:in-out_edge}. The colors of the angles $\alpha$,
  $\beta$ and $\gamma$ must be pairwise different by
  \reflemma{lemma:edge_color}. Hence, $\alpha$ and $\delta$ must have
  the same color (\reflemma{lemma:angle_color}). But $\alpha$ has the
  same color as $e_i$. From \reflemma{lemma:chordal_edge} it follows
  that the alternating cycle $(e_i, F_{\infty})$,$(w,F)$ is
  monochromatic -- contradiction.
\end{proof}

\def\sr{{\scriptstyle\color[rgb]{1,0,0}r}}
\def\sg{{\scriptstyle\color[rgb]{0,.5,.5}g}}
\def\sb{{\scriptstyle\color[rgb]{0,0,1}b}}
   \calc_figscale{100}
    \begin{figure}[htb]
    \centerline{\input{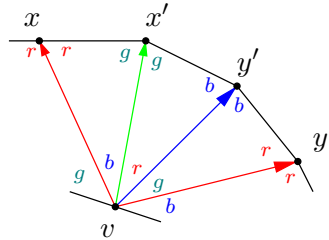}}
    \caption{No two outgoing edges can
 have the same color.\label{fig:2_outgoing}}
    \end{figure}
    

\begin{prop}
\label{prop:2_out}
No two outgoing edges from a vertex have the same color.
\end{prop}

\Proof
Suppose that $v$ has two outgoing edges of the same color
and that all triangles are clockwise.
From the coloring of angles it follows that
edges sharing an angle have different colors.
Even more, the colors of the outgoing edges at $v$ in clockwise order
cycle through 1,2,3. Hence, we find a sequence $x,x',y',y$ of vertices, such that
$vx$ and $vy$ have the same color, see \reffig{fig:2_outgoing}.
Now, $v$ is above $x$ in colors green and blue, so any face incomparable to $x$
which contains $v$ has to be below $x$ in red. The same is true for
$y$. In particular $x> \{v,y',y\}$ and $y > \{v,x,x'\}$ in red. This is
a monochromatic red alternating cycle -- contradiction.
\qed

\begin{cor}
No vertex has four or more outgoing chordal edges. 
\end{cor}

We say that the colors of the chordal edges bounding a face are \emph{the colors
of the face}. A face with two colors is called \emph{bicolored}.

\begin{prop}
\label{prop:bicolor}
No two bicolored faces have the same colors. 
\end{prop}

\begin{proof}
  Suppose $F$ and $F'$ are two such faces. Suppose the two colors are
  red and green. Then $F_{\infty}$ is below $F$ and $F'$ in red and
  green (\reflemma{lemma:chordal_edge}). Therefore, $F$ and $F'$ can
  be below any vertex only in blue. Let $x$ be vertex in $F\setminus F'$
  and $y$ be a vertex in $F'\setminus F$. Then $(x,F')$, $(y,F)$ is a
  monochromatic blue alternating cycle -- contradiction.
\end{proof}

We say that an oriented coloring satisfying
\reflemma{lemma:same_orient}, \refprop{prop:2_out} and
\refprop{prop:bicolor} is \emph{permissible}. The map of
\reffig{fig:canonical_map} is shown with a permissible coloring.
We call it the \emph{canonical map}. The vertices (edges) in the top of the figure
are called $q$-vertices (edges) and the ones in the bottom of the
figure are called $p$-vertices. Note that faces $F_2$, $F_4$ and $F_6$
are bicolored.

   \calc_figscale{140}
    \begin{figure}[htb]
    \centerline{\input{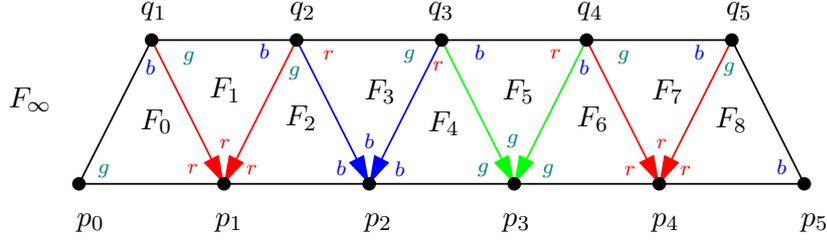}}
    \caption{The canonical map.\label{fig:canonical_map}}
    \end{figure}
    

\begin{lemma}
\label{lemma:canonical_map_construct}
Any maximal path-like map with a permissible coloring of
the chordal edges can be constructed from the canonical map by a sequence of
the following three operations:
\begin{itemize}
\item[(i)] Contracting one of the $q$-edges $\{q_2,q_3\}$ or $\{q_3,q_4\}$.

\item[(ii)] Subdividing a $q$-edge. Chordal edges between the new
  vertices and the $p$-vertex in the triangle are inserted with the
  same color and orientation as the old edges.

\item[(iii)] Deleting all the vertices, edges and faces on one side of
  a chordal edge.

\end{itemize}
\end{lemma}

\begin{proof}
Let $M$ be a maximal path-like map with a permissible coloring of the
chordal edges. If $M$ has $n$ sinks, then there are $n-1$ bicolored
faces in $M$. From \refprop{prop:bicolor} and
\reflemma{lemma:same_orient} it follows that the canonical map has the
maximum possible number of sinks.

Again from  \refprop{prop:bicolor} it follows that $M$ has at most one vertex
with outdegree 3. We can contract $\{q_2,q_3\}$ or $\{q_3,q_4\}$ in the
canonical map to get such a vertex. Note that contracting  $\{q_2,q_3\}$ and
$\{q_3,q_4\}$ would yield a map with a vertex of out-degree four, i.e., 
a map that has no permissible coloring.

Any sinks or 
sources in the canonical map
that are not in $M$ can be removed using operation
(iii). Hence, what remains is possibly to add some sources of degree
one to the map $M'$ we have constructed. But this is easy, since the
new inserted chordal edges must have the same colors as the faces they 
split, i.e., we can just subdivide the $q$-edges as in~(ii).
\end{proof}

\begin{lemma}
\label{lemma:canonical_map_dim}
Let $M$ be the canonical map. Then $\dim(\vef{M})=3$.
\end{lemma}
\begin{proof}
Consider the vertical symmetry through $q_3$ and the edge $\{p_2,p_3\}$.
This partitions the vertex set into a left part $V_l=\{p_0,p_1,p_2,q_1,q_2\}$
a right part $V_r=\{p_5,p_4,p_3,q_5,q_4\}$ and
$\{q_3\}$.

We construct three partial orders, one for each color. We start with
the order of the vertices. In the green order we want $V_r > q_3 >
V_l$.  The ordering of the vertices in $V_l$ and $V_r$ is such that it
allows to reverse all critical pairs corresponding to green angles, i.e., on
$V_l:$ $p_0 > q_1 > q_2 > \{p_2,p_1\}$ and on $V_r:$ $p_3 > q_4 > q_5 >
\{p_4,p_5\}$.  We add the relations $p_2>p_1$ and $p_5>p_4$ so that
the order on $V_l$ and $V_r$ conforms with the clockwise ordering
around the outer face with start in $p_0$ and $p_3$, respectively.

The blue partial order is created symmetrically.  I.e., it is obtained
from the green order with the mappings $p_i \to p_{5-i}$ and $q_j \to
q_{6-j}$. Note that the order on $V_l$ and $V_r$ conforms with the
counterclockwise ordering around the outer face with start in $p_2$
and $p_5$, respectively.

For the red partial order, we construct two linear orders on $V_l \cup
\{q_3\}$ and $V_r \cup \{q_3\}$. These linear orders gives us a partial
order on $V_l \cup V_r \cup \{q_3 \}$. In the
linear order on $V_l \cup \{q_3\}$, the vertices come in the clockwise
ordering around the outer face boundary with $p_1$ as maximal
element, i.e, $p_1>p_0>q_1>q_2>q_3>p_2$. The right part is done symmetrically,
$p_4>p_5>q_5>q_4>q_3>p_3$.

We now have three partial orders on the vertices. We extend these to
partial orders on the vertices, edges and faces in three steps.
First, we insert the Hamilton cycle edges and the outer face as low
as possible in each of the three orders. Then the chordal edges are put above
the outer face in their color, and as low as possible in the other two
colors. Finally, the interior faces are inserted as low as possible.

\Claim R.
Every critical pair is reversed in one of the partial orders. 
\smallskip

The lemma clearly follows from this claim; any three linear extensions
of the partial orders constructed will then form a realizer.

There are three types of critical pairs: edge-face pairs, vertex-edge
pairs and vertex-face pairs. All edge-face pairs are of the form
(chordal edge, outer face), so they are reversed in the color of the edge. 

Consider a vertex-edge critical pair $(v,e)$. For the pair to be critical 
$v$ and $e$ must belong to a triangle and $e$ has to be an edge of the Hamilton cycle.
Such a critical pair corresponds to a colored angle at $v$. 
Since the order of each color reverses all critical pairs corresponding to this
color each critical pair of this class is reversed.

It remains to prove that all vertex-face pairs $(v,F)$ are
reversed. Note that $F$ is an interior face.
If $v \in V_l$ and $F\subset V_r\cup\{q_3\}$, then $(v,F)$ is
reversed in blue. Similarly, if $v \in V_r$ and $F\subset V_l\cup\{q_3\}$, 
then $(v,F)$ is reversed in green. All the
vertices that are incomparable to $F_4$ are above $F_4$ in red, and
$q_3$ is above all incomparable faces in green or blue. Hence,
we only have to show that $(v,F)$ is reversed when $F$ and $v$ are
either both left or both right.

Suppose $v \in V_l$. 
The critical pairs $(v,F_0)$ and $v=p_2,q_2$ are reversed in blue.
The pair $(p_0,F_1)$ is reversed in green and $(p_2,F_1)$ in blue.
The two pairs involving $F_2$ are reversed in green and all three pairs 
with $F_3$ in red. The cases where $v \in V_r$ are symmetric.
\end{proof}

\begin{prop}
\label{prop:path_like_dim}
Let $M$ be a maximal path-like map. Then
$\dim(\vef{M}) = 3$ if and only if there is a permissible coloring of
the chordal edges. 
\end{prop}
\begin{proof}
If $\dim(\vef{M}) = 3$, then there is a permissible coloring  of
the chordal edges. In view of \reflemma{lemma:canonical_map_dim}
we only have to prove that the operations of
\reflemma{lemma:canonical_map_construct} do not increase the dimension.

\textbf{(i):} By symmetry we only have to consider the contraction of
$\{q_2,q_3\}$. The new merged vertex $q_{2,3}$ takes the place of
$q_2$ in green and blue, and the place of $q_3$ in red.

The vertex-edge pairs involving $q_{2,3}$ are reversed in
green and red. Now, $q_{2,3}$ is only below the old position of $q_2$
in red. But the only vertex-face critical pair with $q_2$ that was
reversed in red is $(q_2,F_4)$, and $q_{2,3} \in F_4$, so all
critical pairs involving $q_{2,3}$ are reversed.

Now, the position of a face in a partial order can only change if it
contains $q_2$ or $q_3$ as its highest vertex. The new vertex
$q_{2,3}$ is as high as $q_2$ in green and blue, and as high as $q_3$
in red, so the only affected face is $F_5$ in blue. In the blue
partial order, the critical pairs $(v, F_5)$, $v \in \{p_0, p_1,
q_1\}$, are not reversed anymore. This is taken care of by moving
$q_1$ (and hence $p_1$ and $p_0$) above $q_4$ in red. Since none of
$p_0$, $p_1$ and $q_1$ were comparable to $q_4$ in red before, all
previously reversed critical pairs are still reversed.

\textbf{(ii):} The partial orders are constructed as before (with
possible changes resulting from a $q$-edge contraction). By an argument
similar to the proof of \reflemma{lemma:canonical_map_dim}, $\dim(\vef{M})=3$.

\textbf{(iii):} The only incidence that changes among the remaining
elements of $\vef{M}$ is that one chordal edge $e$ now is on the outer
face. The edge $e$ is moved below $F_{\infty}$ in its color. The only
new critical pair is $(v,e)$, where $v$ is the vertex in the same
interior face as $e$ that is not in $e$. But previously, there was a
critical pair $(v,F)$, $e \in F$, which was reversed, so $(v,e)$ must
be reversed.
\end{proof}

We summarize the result in a theorem:
\begin{thm}
\label{thm:path_like_dim}
Let $M$ be a maximal path-like map. Then the following three conditions are equivalent:
\Item(i) $\dim(\vef{M}) = 3$
\Item(ii) There is a permissible coloring of the chordal edges of $M$. 
\Item(ii) $M$ can be obtained from the canonical map via the operations in 
\reflemma{lemma:canonical_map_construct}.
\end{thm}

\subsection{Algorithmic aspects}\label{sec:alg}

\refthm{thm:path_like_dim} can easily be turned into an algorithm for
testing if $\dim(\vef{M}) \leq 3$ for a maximal path-like map $M$. 
Start by fixing one of the two possible 
orientations and a color some chordal edge. This induces an angle
coloring in the adjacent triangles (\reflemma{lemma:edge_color},
\reflemma{lemma:angle_color}). \reflemma{lemma:same_orient} now gives
us the colors of the angles in all 
the interior triangles in $M$, which in turns induces an oriented
coloring of the chordal edges. Hence, given a fixed orientation of one
chordal edge, any permissible coloring is unique up to permutations of
the colors. To test for $\dim(\vef{M}) \leq 3$, we  check if
any vertex has four outgoing edges or if any two bicolored faces have
the same colors. This can be done in linear time.

Once we have a permissible coloring of the chordal edges of $M$
a  3-realizer can be generated: Since we now know which
vertices are sinks, we know the $p$-edges and $q$-edges of the 
colored map $M$ and can identify the operations of
\reflemma{lemma:canonical_map_construct} (contracting a $q$-edge,
subdividing a $q$-edge and removing a part of the map) that have to be 
applied to get $M$ from the canonical map. The proof of
\reflemma{lemma:canonical_map_dim} gives us a 3-realizer  of the
canonical map, and the proof of \refprop{prop:path_like_dim}
demonstrates how to modify the 3-realizer for each of the
operations. This yields an algorithm to produce a 3-realizer of
$\vef{M}$ from
a permissible coloring of $M$. It is clear from the proofs of
\reflemma{lemma:canonical_map_dim} and \refprop{prop:path_like_dim} that
the running time of this algorithm can be bounded by some constant
times the number of elements in $\vef{M}$. Since $M$ is a planar map,
we have the following theorem. 

\begin{thm}
\label{thm:3-realizer}
There is an algorithm running in time $O(n)$, which takes as input a
maximal path-like map $M$ with $n$ vertices and either 
 returns a 3-realizer of $\vef{M}$, or certifies
 that $\dim(\vef{M}) = 4$.
\end{thm}

\section{Vertex-face posets of dimension at most 3}
\label{sec:vf}

From \refthm{thm:k4_lower_bound} we know that if $\dim(\vf{M}) \leq 3$, then $M$ does
not contain a subdivision of $K_4$. \reffig{fig:bt_example} shows an example 
of a planar map which contains no $K_4$-subdivision but still $\dim(\vf{M}) = 4$.
This example from~\cite{BT97}, has a dual map which is a
$K_{2,3}$, with each edge replaced by a 2-face. More generally every 
map $M$ where we can find fences of vertices and faces like 
in the proof of \refthm{thm:k23_lower_bound} must have $\dim(\vf{M})=4$.

   \calc_figscale{100}
    \begin{figure}[htb]
    \centerline{\input{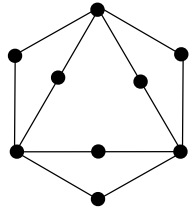}}
    \caption{A planar map with vertex-face poset dimension 4.\label{fig:bt_example}}
    \end{figure}
    

However, unlike in the vertex-edge-face case, there are no 2-connected
maps $M$ of dimension 4 such that both $M$ and $M^*$ are weakly outerplanar.

\begin{thm} \label{thm:vf_dim_upper_b_2c}
Let $M$ be a simple 2-connected planar map such both $M$ and its dual $M^*$ are weakly 
outerplanar. Then $\dim(\vf{M}) \leq 3$.
\end{thm}

\begin{proof}
We may assume that no two vertices $v,w$ of degree 2 are neighbors in
$M$: Two such vertices correspond to elements with identical
comparabilities in $\vf{M}$, so called twins. Hence,
contracting the edge $vw$ in the graph does not affect
$\dim(\vf{M})$.

From \refcor{cor:k23-free}, we know that if all chordal edges of $M$
are moved inside the Hamilton cycle, the dual of the resulting map $M'$ is
$K_{2,3}$-subdivision free. Hence, $M'$ is a path-like map. We will inductively
construct two linear extensions, $L_1$ and $L_2$, of $\vf{M}$, in which
all vertex-face pairs are reversed. Similar to the proof of
\refprop{prop:pathlike}, we start with the path-like map $M_0=M'$ and then
move the required chordal edges outside the Hamilton cycle one by one,
creating a series of maps $M_0,M_1,\ldots,M_k = M$.

The case $M_0=K_3$ is trivial, otherwise, the path-like map $M_0$ has
exactly two faces that contains only one chordal edge. Each of these
two faces contains a vertex of degree 2. Let these two vertices be
$\ell$ and $r$.

Given a face $F$ and a vertex $x$, if there is an $x$--$\ell$ path avoiding $F$,
then we say that {\em $x$ is left of $F$}. Symmetrically, {\em $x$ is right of $F$}
if there is an $x$--$r$ path avoiding $F$. 
This definition of left of and right of coincides with the
intuition of left and right based on a drawing 
where the Hamilton cycle is a circle, and $\ell$ and $r$ are its left
and right extreme points. 
Note that $x$ is neither  right, nor left of $F$ if and only
if $x \in F$. We define a vertex to be to the left (right) of a chordal edge
in the same way.

Next, we inductively construct two linear extensions $L_1^i$ and $L_2^i$ of the
vertex-face poset of $M_i$, for $i=0,1, \ldots, k$, such that
$L_1^k = L_1$ and $L_2^k = L_2$. Since $M_0$ is path-like, the
interior dual of $M_0$ is a path. In the 
linear extensions $L_1^0$ and $L_2^0$ we order the interior faces by
their position in this path, with the face
containing $\ell$ highest in $L_1^0$ and the face containing $r$
highest in $L_2^0$. The vertices and the outer face are 
inserted as high as possible.

Let $e$ be the chordal edge that is moved outside the Hamilton cycle when
$M_i$ is changed to $M_{i+1}$. Before it is moved, it is contained in the two faces 
$F'$ and $F''$ that are inside the Hamilton cycle. Let $F'$ be the left one, i.e.,
let $F'$ contain a vertex $u$ that is left of $e$. When $e$ is moved outside,
some face $G$ outside the Hamilton cycle is split into two faces $G'$ and $G''$.
Let $G'$ contain~$u$. The faces $F'$ and $F''$ are merged 
into a new face $F^+$ (see \reffig{fig:new-vertex_k23}).

   \calc_figscale{78}
    \begin{figure}[htb]
    \centerline{\input{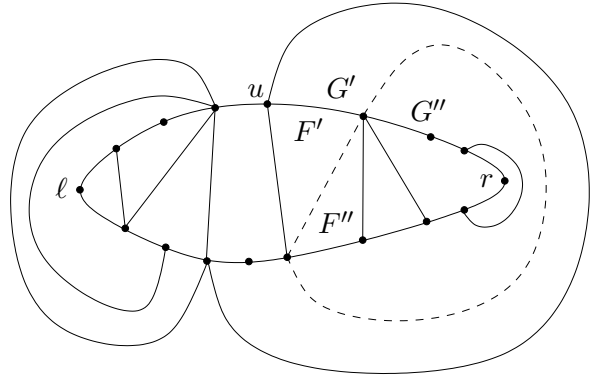}}
    \caption{Moving the dashed edge from the inside to the outside.\label{fig:new-vertex_k23}}
    \end{figure}
    

We can now create $L_1^{j+1}$ and $L_2^{j+1}$ from $L_1^{j}$ and
$L_2^{j}$. In $L_1^{j+1}$, $F^+$ is inserted in the position of $F'$ in
$L^j_1$, $G'$ is inserted in the position of $G$ and $G''$ is inserted
in the position of $F''$. In $L_2^{j+1}$, $F^+$ is inserted in the
position of $F''$ in $L_2^j$, and $G''$ and $G'$ are inserted in the
positions of $G$ and $F'$, respectively.

\Claim A. For $i=0,1,\ldots,k$, $L_1^i$ and $L^i_2$ are linear extensions of $\vf{M_i}$. 
If a vertex $v$ is to the left of a face $F$ in the map $M_i$,
then $v > F$ in $L_1^i$, and if $v$ is to the right of $F$, $v > F$ in $L_2^i$.
  \smallskip

  \ni
  The claim can be verified by induction on $i$. From the construction of
  $L_1^0$ and $L_2^0$, it is clear that the claim is true for $i=0$.
  
  Suppose the claim is true for $i$. The map $M^{i+1}$ is
  constructed by moving the edge $e$ in $M_i$ outside the Hamilton
  cycle. To verify that $L_1^{i+1}$ and $L_2^{i+1}$ are linear 
  extensions of $\vf{M_{i+1}}$ it is enough to check that $F$, $G'$ and
  $G''$ are above all the vertices contained in them. This
  is immediate from the construction.

  It remains to prove the second part of the claim. 
  By induction and symmetry it is enough to consider the case where
  $v$ is to the left of $e$ and $F$ is one of $F^+$, $G'$ and $G''$.
  If $v$ is to the left of $e$, either $v \in F'$ or $v$ is to the left of
  $F'$. In the latter case $v$ is above $F'$ in $L_1^i$, so $v > F^+$ in $L_1^{i+1}$.
  If $v$ is to the left of $G$, then $v$ is also to the left of $G'$
  and $G''$. Since $v > G$ in $L_1^{i}$, $v > G' > G''$ in
  $L_1^{i+1}$. On the other hand, if $v \in G$, $v$ must also be in
  $G'$ and to the left of $G''$. But $v$ is to the left of $F''$, so
  $v > F''$ in $L_1^i$ by construction, and hence $v > G''$ in
  $L_1^{i+1}$. Hence, the claim is true for $i+1$.
  \qedclaim

Claim A implies that all vertex-face critical pairs of $\vf{M}$ are
reversed in $L_1$ and $L_2$.  It remains is to find a linear extension
$L_3$ of $\vf{M}$ which reverses all vertex-vertex and face-face
critical pairs. Such a linear extension can be obtained by taking all
faces above all vertices. The order of the faces in $L_3$ is choosen
as the reverse of the order of faces in $L_1$ and alike the order of
vertices is reversed between $L_3$ and $L_1$. Together $L_1,L_2,L_3$
reverse all critical pairs. Hence, they form a realizer and
$\dim(\vf{M})\leq 3$.
\end{proof}

In a strongly outerplanar map $M$ it is never the case that we can find fences 
like in the proof of \refthm{thm:k23_lower_bound}. The interior dual is a tree, 
and $F_{\infty}$ contains all the vertices. Hence, $F^*_{\infty}$ has to be one of the 
degree 3-vertices of any $K_{2,3}$-subdivision in $M^*$ ($M$ contains
no $K_{2,3}$-subdivision since it is outerplanar). 
Therefore, the existence of a strongly outerplanar map of vertex-face dimension~4 
is not obvious.

   \calc_figscale{100}
    \begin{figure}[htb]
    \centerline{\input{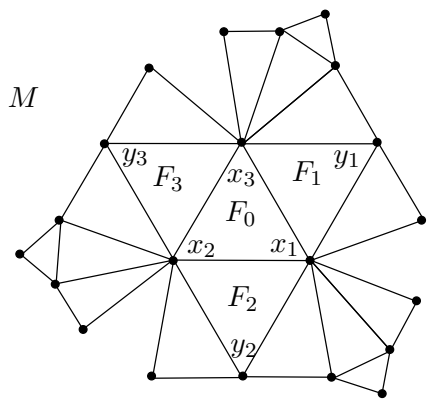}}
    \caption{An outerplanar map with $\dim(\vf{M}) = 4$.\label{fig:vf_graph}}
    \end{figure}
    

 \begin{thm} \label{thm:vf_dim4}
 The strongly outerplanar map $M$ shown in \reffig{fig:vf_graph} has a vertex-face poset of dimension
 four.
 \end{thm}
 \begin{proof}

Suppose $\dim(\vf{M}) \leq 3$. Again we identify the three linear
extensions of a realizer with three 
colors and use these to color critical pairs. Our main focus will be
on the coloring of critical pairs 
involving a vertex and an interior face. 

\Fact A. 
For an interior face $F$ only two colors appear at critical pairs~$(v,F)$.\smallskip

\ni
Suppose that the critical pair $(F,F_{\infty})$ is reversed in color
$i$. This forces all vertices below $F$ in color $i$, and hence all
critical pairs $(v,F)$ are reversed in the other two colors.
\qedclaim

\Fact B. 
If a triangular face $\Delta=\{v_1,v_2,v_3\}$ is surrounded by
interior faces $F_1,F_2,F_3$ such that $v_i \not\in F_i$, then the
three critical pairs $(v_i,F_i)$ use all three colors.
\smallskip

\ni 
Any two of the three critical pairs form an alternating cycle and,
hence, require different colors.  Equivalently, the order induced by
$v_1,v_2,v_3,F_1,F_2,F_3$ is a 3-crown.  
\qedclaim

These two facts are applied to the central face $F_0$ of the map $M$:
Fact B implies that the three critical pairs $(x_i,F_{i-1})$ use all
three colors.  Symmetry among the colors allows us to assume that
$(x_2,F_1)$ is red, $(x_3,F_2)$ green and $(x_1,F_3)$ blue. Fact A
implies that two of the three critical pairs $(y_i,F_0)$ have the same
color. The symmetry of the graph allows to assume that this duplicated
color is blue.  It is infeasible to have $(y_3,F_0)$ in blue, because
it forms an alternating cycle with the blue pair $(x_1,F_3)$.  Hence,
$(y_1,F_0)$ and $(y_2,F_0)$ are both blue.

To reach a contradiction we can from now on concentrate on the
submap of $M$ shown in \reffig{fig:vf_steps}. The colors of critical
pairs which have already been fixed are indicated by
the colored arrows in the figure.

   \calc_figscale{100}
    \begin{figure}[htb]
    \centerline{\input{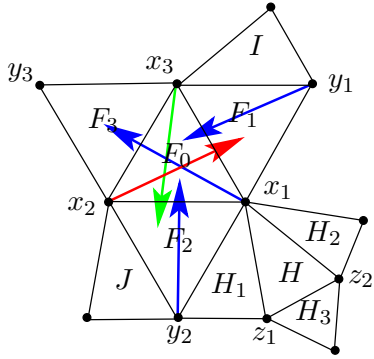}}
    \caption{Having fixed the colors of some critical pairs we concentrate on a submap $M'$ of $M$.\label{fig:vf_steps}}
    \end{figure}
    

To avoid a monochromatic alternating cycle with the blue pair $(y_1,F_0)$ and
with the green pair $(x_3,F_2)$ the color of $(x_1,I)$ has to be red.
Similarly, the colors of $(y_2,F_0)$ and $(x_2,F_1)$ imply that $(x_1,J)$ is green.

From the critical pairs $(x_1,J)$ and $(x_1,F_3)$  we know that $x_1 > x_2$ in green
and blue, so all critical pairs $(x_2,F)$, where $x_1 \in F$, must be
red. Similarly, all critical pairs $(x_3,F)$, where $x_1 \in F$, must be
green. In particular we have $(x_2,H_i)$ red and $(x_3,H_i)$ green for $i=1,2$.

From Fact A applied to $H_1$ and $H_2$ we can conclude that neither 
$(z_1,H_2)$ nor $(z_2,H_1)$ can be blue. Applying Fact B to face $H$
we can conclude that $(x_1,H_3)$ is blue.
See \reffig{fig:vf_more}.

   \calc_figscale{100}
    \begin{figure}[htb]
    \centerline{\input{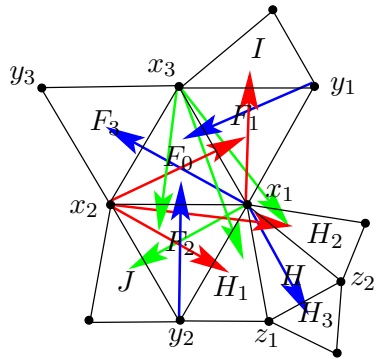}}
    \caption{Colored critical pairs on $M'$.\label{fig:vf_more}}
    \end{figure}
    

Consider the critical pair $(z_1,F_0)$. It forms an alternating cycle with
 $(x_1,H_3)$, hence it can't be blue. It forms an alternating cycle with
 $(x_2,H_1)$, hence it can't be red. It forms an alternating cycle with
 $(x_3,H_1)$, hence it can't be green. Consequently there is no legal
3-coloring of the hypergraph of critical pairs of $\vf{M}$.
\end{proof}

The maximal outerplanar map $T_4$ shown in \reffig{fig:vf_graph_sym}
has a vertex-face poset of dimension 3 (a 3-realizer is listed in
Table~\ref{table:realizer}).  This suggests that the example of a
strongly outerplanar map $M$ with $\dim(\vf{M}) = 4$ given in
\refthm{thm:vf_dim4} is close to a minimal example.
\reffig{fig:vf_non_2con} shows a map where all 2-connected components
are submaps of $T_4$ and hence have vertex-face poset dimension
3. Still an argument as in \refthm{thm:vf_dim4} shows that the map in
\reffig{fig:vf_non_2con} has a 4-dimensional vertex-face poset.

   \calc_figscale{100}
    \begin{figure}[htb]
    \centerline{\input{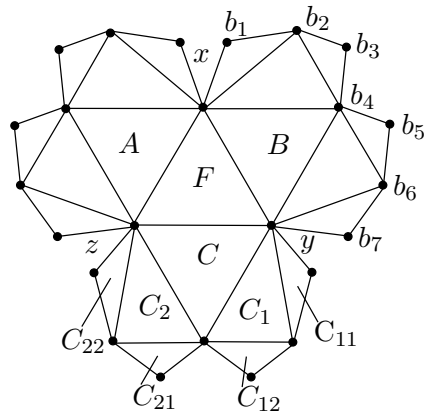}}
    \caption{A 2-connected outerplanar map with vertex-face
  dimension 3. The naming scheme of the faces and vertices of the map
  is indicated in the figure.\label{fig:vf_graph_sym}}
    \end{figure}
    

\begin{table}[h!] \label{table:realizer}
\begin{tabular}{ccc@{\hspace{2cm}}ccc@{\hspace{2cm}}ccc}

$L_1^1$        &             $L_2^1$         &           $L_3^1$          &          $L_1^2$          &          $L_2^2$           & $L_3^2$         & $L_1^3$        & $L_2^3$         & $L_3^3$         \\         
\hline                                                                                                                                                                                  
$A_{22}$     &              $C$          &            $C_{11}$      &          $B_{2}$        &           $F$            &  $c_7$        & $c_7$        &  $a_1$        &  $b_7$        \\         
$A_{2}$      &              $B_{12}$     &            $C_{1}$       &          $b_4$          &           $C_{11}$       &  $a_2$        & $C_{2}$      &  $A_1$        &  $B_{2}$      \\         
$A$          &              $B_{21}$     &            $C_{12}$      &          $B_{22}$       &           $c_1$          &  $a_1$        & $C$          &  $A$          &  $B$          \\         
$F$          &              $B_{2}$      &            $C_{21}$      &          $b_6$          &           $C_1$          &  $a_3$        & $z$          &  $z$          &  $y$          \\         
$B$          &              $B_{22}$     &            $C_{2}$       &          $b_7$          &           $y$            &  $a_5$        & $C_{21}$     &  $A_{12}$     &  $B_{21}$     \\         
$B_{1}$      &              $F_{\infty}$ &            $C_{22}$      &          $A_{21}$       &           $C_{12}$       &  $C$          & $c_6$        &  $a_2$        &  $b_6$        \\         
$B_{11}$     &              $b_6$        &            $A_{11}$      &          $a_6$          &           $c_2$          &  $c_4$        & $c_5$        &  $a_3$        &  $b_5$        \\         
$F_{\infty}$ &              $b_3$        &            $A_{1}$       &          $a_5$          &           $c_3$          &  $A_{22}$     & $C_{12}$     &  $A_{21}$     &  $B_{12}$     \\         
$x$          &              $b_5$        &            $A_{12}$      &          $A_{12}$       &           $C_{21}$       &  $a_7$        & $c_3$        &  $a_5$        &  $b_3$        \\         
$a_7$        &              $b_7$        &            $A_{21}$      &          $a_3$          &           $c_5$          &  $A_{2}$      & $C_{1}$      &  $A_2$        &  $B_1$        \\         
$b_1$        &              $B_{11}$     &            $F_{\infty}$  &          $A_1$          &           $C_2$          &  $a_6$        & $c_4$        &  $a_4$        &  $b_4$        \\         
$B_{12}$     &              $b_1$        &            $c_2$         &          $a_4$          &           $c_4$          &  $A$          & $C_{11}$     &  $A_{22}$     &  $B_{11}$     \\         
$b_2$        &              $B_1$        &            $c_6$         &          $A_{11}$       &           $C_{22}$       &  $a_4$        & $y$          &  $x$          &  $x$          \\         
$b_3$        &              $b_2$        &            $c_1$         &          $a_2$          &           $c_6$          &  $F$          & $c_2$        &  $a_6$        &  $b_2$        \\         
$B_{21}$     &              $B$          &            $c_3$         &          $a_1$          &           $c_7$          &  $z$          & $c_1$        &  $a_7$        &  $b_1$        \\         
$b_5$        &              $b_4$        &            $c_5$         &          $C_{22}$       &           $A_{11}$       &  $B_{22}$     &              &               &               \\

\end{tabular}
\caption{A 3-realizer of the map in \reffig{fig:vf_graph_sym}. The order $L_i$ is 
obtained from the concatenation $L_i = L_i^1 \oplus L_i^2 \oplus L_i^3$.}
\end{table}

   \calc_figscale{80}
    \begin{figure}[htb]
    \centerline{\input{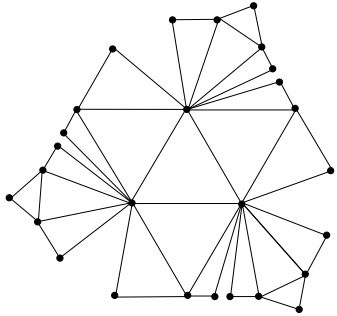}}
    \caption{A map $M$ with $\dim(\vf{M}) = 4$,
where each 2-connected component~$C$ has $\dim(\vf{C}) = 3$\label{fig:vf_non_2con}}
    \end{figure}
    

\section{Concluding remarks}

When we started our investigations we set out to characterize the
planar maps with vertex-edge-face posets of dimension at most 3. We proved
that for all such maps $M$, both $M$ and its dual $M^*$ must be weakly
outerplanar. In the case of maximal path-like maps, we found
necessary and sufficient conditions for dimension at most 3 using an
oriented coloring of the chordal edges. What remains open are four
cases:

\Bitem $M$ is path-like but not maximal, i.e., has non-triangular faces.

\smallskip\ni
  If $M$ can be obtained from a maximal path-like map $M_0$ by
  subdividing some edges of the Hamilton cycle, then $\dim(\vef{M}) =
  \dim(\vef{M_0})$.  To see this, consider a 2-connected map $M$ 
  with $\dim(\vef{M}) \leq 3$ and an edge $e = \{u,v\}$ of the Hamilton cycle.
  Let edge $e$ be subdivided in the new map $M'$ by a vertex $w$ incident to 
  the new cycle edges $e_1 = \{u,w\}$ and $e_2 = \{w,v\}$. We
  can modify a 3-realizer of $\vef{M}$ to obtain realizer of $\vef{M'}$ 
  in the following way. In each linear extension,
  insert $e_1$ in the old position of $e$ if $v < u$, and right
  below $v$ if $v > u$. Symmetrically, $e_2$ is inserted into the
  position of $e$ if $u < v$, and right
  below $u$ if $u > v$. The vertex $w$ is inserted right below
  $\min\{e_1, e_2\}$. This produces a 3-realizer of $\vef{M'}$. 
  The general case remains a challenge.

\Bitem $M$ is simple, and both $M$ and $M^*$ are weakly outerplanar, but neither of
  them is strongly outerplanar.

\smallskip\ni
  We can extend the coloring approach from
  \refsection{sec:path_like} to get a set of necessary conditions for
  $\dim(\vef{M})\leq 3$ in this case. There is a natural way to get an
  oriented coloring as in the case of 
 path-like maps, if $M$ is a different drawing of the graph of some maximal path-like map.
 Instead of coloring the angles of triangular faces, we
 color the angles of triangles in the strongly outerplanar drawing of
 the graph of $M$. 
 Again, this angle coloring can be encoded as
 an oriented coloring of the chordal edges. Instead of each vertex
 being a sink or a source, each vertex will now be a sink on one side
 of the Hamilton cycle and a source on the other side. The proof of
 this is similar to the proof of \reflemma{lemma:sink_or_source}.

If $F$ is a face in the 2-connected map $M$ with $\dim(\vef{M}) \leq
3$, the submap $M_F$ induced by the vertices in $F$ must be path-like
by \refthm{thm:k23_lower_bound}. In the same way, the submap
$M_{v^*}^*$ of the dual map $M^*$ induced by the dual vertices in the
dual face $v^*$ is also path-like.  Since $(x,y)$ is a critical pair in
$\vef{M}$ if and only if $(y^*,x^*)$ is a critical pair in
$\vef{M^*}$, the primal oriented coloring induces an oriented coloring
in the dual map. Hence, the oriented coloring of a map $M$ with
$\dim(\vef{M}) \leq 3$ must be permissible ``locally'' around each
vertex and face. The question remains whether the existence of such a
locally permissible coloring is also a 
sufficient condition for $\dim(\vef{M}) \leq 3$, or if there are some 
non-local effects that force $\dim(\vef{M})=4$.

\Bitem $M$ is not simple.

\Bitem $M$ is not 2-connected.

\smallskip\ni
  Suppose $M$ is not 2-connected. From $\dim(\vef{C}) \leq 3$ for each
  2-connected component $C$ it can not be concluded that
  $\dim(\vef{M}) \leq 3$. The conclusion is not even possible if all
  components are maximal path-like maps and have a common outer face.
  Consider the map $M$
  constructed by taking two maximal path-like maps $C_1$ and $C_2$ and
  identifying two vertices $v_1 \in C_1$ and $v_2 \in  C_2$ and the
  outer faces of each map. We choose $C_1$ and $C_2$ such that
  $\dim(\vef{C_1}) = \dim(\vef{C_2}) = 3$ and that
  in any permissible coloring of the chordal edges in each map $C_i$ there
  will be two outgoing edges from $v_i$. Such maps
  clearly exist. A straightforward modification of \refprop{prop:2_out}
  will now show that $\dim(\vef{M}) = 4$.

\bigskip
\ni
{\bf\large Vertex-face posets and posets of height 2}
\medskip

For vertex-face posets, we saw that it seems hard to characterize
even the strongly outerplanar 2-connected maps with dimension at most
3. This relates to the long-standing open question if it is
NP-hard to determine if the dimension of a height 2 poset is at most~3. 
Yannakakis \cite{Yannakakis82} proved in 1982 that
it is NP-hard to determine if $\dim(\PP) \leq 3$ for posets $\PP$ of height
at least 3.
Brightwell and Trotter \cite{BT97} refined the question and asked if
it is NP-hard to recognize planar maps with $\dim(\vf{M}) \leq 3$. Given our results,
it makes sense to ask this question even for 2-connected strongly outerplanar maps.

\subsection*{Acknowledgment}
We thank Graham Brightwell for stimulating discussions and a referee
for comments that helped improve the paper.

\small
\bibliographystyle{my-siam}
\bibliography{planar_maps}

\end{document}